\documentclass[12pt]{amsart}
\usepackage[utf8]{inputenc}
\usepackage{enumerate, url, graphicx}

\newtheorem{theorem}{Theorem}[section] \newtheorem{lemma}[theorem]{Lemma}
\newtheorem{proposition}[theorem]{Proposition}

\newtheorem{corollary}[theorem]{Corollary}
\theoremstyle{definition}
\newtheorem{definition}[theorem]{Definition}
\newtheorem{example}[theorem]{Example}

\theoremstyle{remark}
\newtheorem{remark}[theorem]{Remark}
\numberwithin{equation}{section}

\title[Extreme values of the derivative of Blaschke products]
{Extreme values of the derivative of Blaschke products and hypergeometric polynomials}

\author{Leonid V. Kovalev}
\address{215 Carnegie, Mathematics Department, Syracuse University, Syracuse, NY 13244, USA}
\email{lvkovale@syr.edu}
\thanks{L.V.K. supported by the National Science Foundation grant DMS-1764266.}

\author{Xuerui Yang}
\address{215 Carnegie, Mathematics Department, Syracuse University, Syracuse, NY 13244, USA}
\email{xyang20@syr.edu}
\thanks{X.Y. supported by Young Research Fellow award from Syracuse University.}



\begin{document}
\baselineskip6.5mm

\begin{abstract} A finite Blaschke product, restricted to the unit circle, is a smooth covering map. The maximum and minimum values of the derivative of this map reflect the geometry of the Blaschke product. We identify two classes of extremal Blaschke products: those that maximize the difference between the maximum and minimum of the derivative, and those that minimize it. Both classes turn out to have the same algebraic structure, being the quotient of two hypergeometric polynomials.
\end{abstract}

\maketitle

\section{Introduction} 

A finite Blaschke product of degree $n$ is a rational function of the form
\[
B(z) = \alpha \prod_{k=1}^n\frac{z-a_k}{1-\overline{a_k}z}
\]
where $a_1,\dots, a_n\in \mathbb{D}$ and $\alpha\in \mathbb T$. Here and below we use the notation $\mathbb D = \{z\in \mathbb C\colon |z|<1\}$ and  $\mathbb T = \{z\in \mathbb C\colon |z|=1\}$. All Blaschke products in this paper are finite. Such products can be characterized as the rational functions that map $\mathbb T$ onto itself and have no poles in $\mathbb D$. Their many connections in complex analysis and operator theory are described in the book~\cite{GarciaMashreghiRoss}. In this paper we focus on the behavior of a Blaschke product on the unit circle $\mathbb T$, and more specifically on the range of its derivative. 

\begin{definition}\label{def:extremal_derivative} Given a finite Blaschke product $B$, let
\begin{equation}
M(B) = \sup_{|z|=1} |B'(z)|, \quad m(B) = \inf_{|z|=1} |B'(z)| .
\end{equation}
\end{definition}

By the maximum principle, $M(B)$ could be equivalently defined as the Hardy space norm $\|B'\|_{H^\infty}$, i.e., the supremum of $|B'|$ on the unit disk $\mathbb D$. The integral Hardy norms of $B'$ were studied in~\cite{GluchoffHartmann}. The maximum and minimum of $|B'|$ on the unit circle are of interest for several reasons.  
\begin{itemize}
    \item If $m(B)>1$, then the restriction of $B$ to $\mathbb T$ is an expanding map, which has several consequences for the dynamics of $B$~\cite{Martin}.
    \item If $m(B) < 2$, then $B$ has nonempty ``thin part'' $\{z\in \mathbb T\colon |B'(z)| < 2\}$; the restriction of $B$ to its thin part extends to a homeomorphism of $\mathbb T$~\cite{McMullen}.
    \item For a Blaschke product of degree $n$, the inequality $m(B)\ge n-1$ is equivalent to the quotient $B(z)/z^{n-1}$ being a homeomorphism of $\mathbb T$. Similarly, $M(B)\le n+1$ is equivalent to $z^{n+1}/B(z)$ being a homeomorphism of $\mathbb T$. These two constructions provide all circle homeomorphisms with a Fourier series that terminates in one direction~\cite{KY}. 
\end{itemize}

An obvious property of the extreme values of $|B'|$ is that $m(B)\le n \le M(B)$, where $n$ is the degree of $B$. Indeed, the restriction of $B$ to $\mathbb T$ is an $n$-fold cover, which implies that the mean value of $|B'|$ on $\mathbb T$ is equal to $n$. 
However, the inequality $m(B)\le n \le M(B)$ does not completely describe the possible pairs $(M(B), m(B))$ for an $n$-fold Blaschke product $B$. Such a description is provided by the following theorem, the second part of which is our main result.  

\begin{theorem}\label{thm:main} (i) For every Blaschke product $B$ of degree $n$ the numbers $M=M(B)$ and $m=m(B)$ satisfy  
\begin{equation}\label{thm:main-ineq}
\frac{n}{M-n+1} \le  m \le  n-1 + \frac{n}{M}      
\end{equation}
and
\begin{equation}\label{thm:main-ineq-trivial}
0<m\le n \le M. 
\end{equation}

(ii) Conversely, for any triple $(n, m, M)\in \mathbb N\times (0, \infty) \times (0, \infty)$ that satisfies~\eqref{thm:main-ineq} and~\eqref{thm:main-ineq-trivial}, there exists a Blaschke product $B$ of degree $n$ for which $M(B) = M$ and $m(B) = m$.  
\end{theorem}

To further understand the relation between $M(B)$ and $m(B)$, we consider the Blaschke products for which equality is attained in~\eqref{thm:main-ineq}. 

\begin{definition}\label{def:extremal-Bp} A Blaschke product $B$ is \emph{extremal} if it attains equality in either part of~\eqref{thm:main-ineq}: that is, either 
\begin{equation}\label{def:extremal-1}
m(B) = \frac{n}{M(B) - n + 1}    
\end{equation}
or 
\begin{equation}\label{def:extremal-2}
m(B) = n-1 + \frac{n}{M(B)}. 
\end{equation}
We sometimes distinguish between ~\eqref{def:extremal-1} and~\eqref{def:extremal-2} as extremal products \emph{of the first kind} and \emph{of the second kind}, respectively.   
\end{definition}

Extremal Blaschke products of the first kind minimize $m(B)$ for a given value of $M(B)$, which means that $|B'|$ attains a wide range of values on $\mathbb T$. In contrast, extremal products of the second kind  maximize $m(B)$ for a given $M(B)$, which means that $|B'|$ stays as close to constant as possible. Figure~\ref{fig:deg15der} in Section~\ref{sec:zeros} illustrates this difference on the example $n=15$ and $M = 20$. But despite these contrasting features, Theorem~\ref{thm:uniqueness} will show that both kinds of extremal Blaschke products have the same algebraic structure.  

Our uniqueness result for extremal Blaschke product describes them in terms of hypergeometric functions. In its statement we  use the notation
\begin{equation}\label{nu-def}
\nu(B) = \begin{cases} M(B) - n,\quad \text{if $B$ is of the first kind} \\
m(B) - n,\quad \text{if $B$ is of the second kind.} \\
\end{cases}
\end{equation}
Equations~\eqref{def:extremal-1}--\eqref{def:extremal-2} show that 
\begin{equation}\label{nu-prop}
\frac{n}{\nu(B)+1} = \begin{cases} m(B),\quad \text{if $B$ is of the first kind} \\
M(B),\quad \text{if $B$ is of the second kind.} \\
\end{cases}
\end{equation}
Hence we always have $\nu(B) > -1$, see~\eqref{def:extremal-2}. The extremal  products with $\nu(B)>0$ are of the first kind, and those with $-1<\nu(B)<0$ are of the second kind. 

\begin{theorem}\label{thm:uniqueness} Suppose $B$ is an extremal Blaschke product of degree $n$. Let $\nu$ be as in~\eqref{nu-def}. If $\nu=0$, then 
$B(z) = \alpha z^n$ for some unimodular constant $\alpha$. If $\nu\ne 0$, then there exist unimodular constants $\alpha, \beta$ such that 
\begin{equation}\label{uniqueness-formula}
B(z) = \alpha \frac{p(\beta z)}{q(\beta z)}    
\end{equation}
where $p$ is the hypergeometric polynomial
\begin{equation}\label{hyper-poly}
p(z) = F(-n, \nu +2;  -n -\nu + 1; z)
\end{equation}
(see Definition~\ref{def:hyper}) and $q$ is the conjugate-reciprocal polynomial of $p$, that is $q(z) = z^n \overline{p(1/\bar z)}$.  
\end{theorem}

\begin{remark} The existence of extremal Blaschke products (Theorem~\ref{thm:main}) implies that the product defined in~\eqref{uniqueness-formula} is indeed extremal. A more detailed description of the Blaschke product~\eqref{uniqueness-formula} is provided by Theorem~\ref{thm:extreme-values}. 
\end{remark}

\begin{remark} Since the polynomial~\eqref{hyper-poly} has real coefficients, one could write $q(z) = z^n p(1/z)$ in Theorem~\ref{thm:uniqueness}, but it is more natural to think of $q$ as conjugate-reciprocal to $p$. 
\end{remark}

Theorem~\ref{thm:main} has a consequence for rational circle homeomorphisms and diffeomorphisms. Recall from~\cite{KY} that the mapping $z\mapsto B(z)/z^{n-1}$ is a homeomorphism of $\mathbb T$ if and only if $m(B)\ge n-1$, and it is a diffeomorphism if and only if $m(B) > n-1$. Similarly for $z^{n+1}/B(z)$: it is a homeomorphism iff $M(B)\le n+1$ and a diffeomorphism iff $M(B) < n+1$.  

\begin{corollary}\label{cor:deg2}  If $B$ is a Blaschke product of degree $2$ and the restriction of $z^3/B(z)$ to $\mathbb T$ is a  homeomorphism, then $B(z)/z$ is also a homeomorphism of $\mathbb T$. Moreover, if  $z^3/B(z)$ is a  circle diffeomorphism, then $B(z)/z$ is a diffeomorphism.
\end{corollary}

\begin{proof} From~\eqref{thm:main-ineq} with $n=2$ we obtain $m\ge 2/(M-1)$. Thus, $M \le 3$ implies $m\ge 1$, and $M < 3$ implies $m>1$. 
\end{proof}

Theorem~\ref{thm:main} also shows that for degrees $n>2$, the inequality $M(B) \le n+1$ does not imply $m(B) \ge n-1$. Thus, for $n > 2$ there exist Blaschke products such that $z^{n+1}/B(z) $ is a circle homeomorphism but $B(z)/z^{n-1}$ is not. 

The paper is structured as follows. The first part of Theorem~\ref{thm:main} is proved in Section~\ref{sec:prelim}. After introducing hypergeometric functions in Section~\ref{sec:hyper} we establish a nonlinear relation between three such functions  (Lemma~\ref{lem:identity}) which appears to be new. This relation unlocks the extremal properties of \emph{hypergeometric Blaschke products} in Section~\ref{sec:hyper-Blaschke}. We complete the proof of Theorem~\ref{thm:main} in Section~\ref{sec:sharpness} and prove the uniqueness of extremal products (Theorem~\ref{thm:uniqueness}) in Section~\ref{sec:uniqueness}. Section~\ref{sec:zeros} concerns the zeros of extremal Blaschke products.  

\section{Preliminaries}\label{sec:prelim} 

We begin the proof of Theorem~\ref{thm:main} with the easy part~(i). It is based on the following lemma, which is proved by considering the residues of a rational function at its simple poles. 

\begin{lemma}\cite[Lemma 4.2]{Ellipses} \label{lem:ellipses} Let $B$ be a Blaschke product of degree $n\ge 2$ with $B(0)=0$. Fix $\lambda\in \mathbb T$ and let $z_1,\dots, z_n\in \mathbb T$ denote the $n$ distinct solutions of $B(z)=\lambda$. Then 
\begin{equation}\label{ellipses1}
    \frac{B(z)}{z(B(z)-\lambda)} = \sum_{j=1}^n \frac{m_j}{z-z_j}
\end{equation}
where $\sum_{j=1}^n m_j=1$ and 
\begin{equation}\label{ellipses9}
\frac{1}{m_j} = |B'(z_j)| = \frac{z_j B'(z_j)}{\lambda}, \quad j=1, \dots, n.
\end{equation}
\end{lemma}

\begin{proof}[Proof of Theorem~\ref{thm:main} (i)] Let $B$ be any Blaschke product of degree $n\ge 1$. Apply Lemma~\ref{lem:ellipses} to the product $\widetilde{B}(z) := zB(z)$, which satisfies $|\widetilde{B}'(z)| = |B'(z)|+1$ for all $z\in \mathbb T$. The lemma shows that for every $\lambda \in \mathbb T$ we have
\begin{equation}\label{ellipses2}
    \sum_{j=0}^{n} \frac{1}{|B'(z_j)| + 1} = 1
\end{equation}
where $z_0, \dots, z_{n}$ are solutions of the equation $zB(z) = \lambda$.

Using the inequality $|B'(z_j)|\le M$ for $j=1, \dots, n$,  we obtain 
\[
\frac{1}{|B'(z_0)|+1} + \frac{n}{M+1} \le 1 .
\]
Hence
\[ |B'(z_0)|\ge \frac{n}{M-n+1} \]
and since $\lambda\in \mathbb T$ was arbitrary, the left hand side of~\eqref{thm:main-ineq} follows. 

To prove the right hand side of~\eqref{thm:main-ineq} we assume $m>n-1$ since the inequality is trivial otherwise. Using $|B'(z_j)|\ge m$ for $j=1, \dots, n$,  we obtain 
\[
\frac{1}{|B'(z_0)|+1} + \frac{n}{m+1} \ge 1 .
\]
Since $\lambda$ was arbitrary, it follows that $M\le  \frac{n}{m-n+1}$, hence $m \le n-1 + n/M$.  
\end{proof}

Seeing that the second part of Theorem~\ref{thm:main} is the converse of the part that was derived from Lemma~\ref{lem:ellipses}, it is natural to ask about the converse of this lemma. The converse indeed holds in the following form.

\begin{lemma}\label{lem:prescribed}
Given any distinct points $z_1, \dots, z_n\in \mathbb T$ and positive numbers $m_1, \dots, m_n$ with $\sum_{j=1}^n m_j = 1$, there exists a Blaschke product $B$ of degree $n$ with $B(0)=0$ such that~\eqref{ellipses1} and~\eqref{ellipses9} hold for some $\lambda \in \mathbb T$.
\end{lemma}

Lemma~\ref{lem:prescribed} goes back to the work of Gau and Wu~\cite[Theorem 3.1]{GauWu} on the numerical range of certain completely nonunitary contractions. A more direct proof was given by Gorkin and Rhoades in~\cite[Theorem 9]{GorkinRhoades}, see also~\cite[Lemma 13]{DaeppGorkinVoss}. 

However, an attempt to use Lemma~\ref{lem:prescribed} to prove the second part of Theorem~\ref{thm:main} is unlikely to succeed. Indeed, given $n, M, m$ as in Theorem~\ref{thm:main}~(ii), one can use Lemma~\ref{lem:prescribed} to construct a Blaschke product $B$ with $B(0)=0$ and $\deg B=n+1$ such that for  $\widetilde{B}(z)=B(z)/z$ it holds that $|\widetilde B'|$ attains the values $M$ and $m$, and only these values, on some  set of the form $B^{-1}(\lambda)$. But this only tells us that $M(\widetilde B)\ge M$ and $m(\widetilde B)\le m$, not that $M(\widetilde B) = M$ and $m(\widetilde B) = m$. Our proof of Theorem~\ref{thm:main}~(ii) is based on a completely different idea. It involves constructing Blaschke products from hypergeometric functions, which are the  subject of next section.

\section{Hypergeometric functions}\label{sec:hyper}

\begin{definition}\label{def:hyper} The \emph{hypergeometric function} $F$ is defined by the power series 
\begin{equation}\label{hyper-series}
F(a, b; c; z) =      
    \sum_{k=0}^\infty \frac{(a)_k (b)_k}{(c)_k} \frac{z^k}{k!}
\end{equation}
where subscripts (the Pochhammer symbol) are understood as rising factorials: $(a)_0=a$ and $(a)_k = a(a+1)\cdots (a+k-1)$ for $k\in \mathbb N$. In general, ~\eqref{hyper-series} is well defined if $1-c\notin \mathbb N$ and $|z|<1$. But in the special case when $a = -n$ for some $n\in \mathbb N$, the series terminates at the index $k=n$, becoming a \emph{hypergeometric polynomial}. In this case we can allow any $z\in \mathbb C$ and any value of $c$ in the set $\mathbb C\setminus \{0, -1, -2, \dots, 1-n\}$. The latter holds because $(c)_k \ne 0$ when $k=0, \dots, n$. 
\end{definition}

The following is a known result (see ~\cite[Theorem 1]{DriverDuren-Indag} and~\cite{Ridley}) but we include a short proof. 

\begin{proposition}\label{zeros-circle} Fix $n\in \mathbb N$ and a real number $\lambda > -1/2$ such that $\lambda\ne 0$. All zeros of the  polynomial $h(z) = F(-n, \lambda; -n+1 - \lambda; z)$ are simple and lie on the unit circle $\mathbb T$. 
\end{proposition}

\begin{proof} The polynomial $h$ is related to the Gegenbauer polynomial $C_n^{(\lambda)}$ by the formula
\begin{equation}\label{Gegenbauer}
C_n^{(\lambda)}(\cos \theta)
= e^{in\theta} \frac{(\lambda)_n}{n!}
h(e^{-2i\theta}), \quad 0\le \theta \le \pi,
\end{equation}
see~\cite[22.3.12]{AbramowitzStegun} or ~\cite[18.5.11]{DLMF}.  
The general theory of orthogonal polynomials~\cite[Theorem 5.4.1]{AndrewsAskeyRoy} shows that $C_n^{(\lambda)}$ has $n$ simple zeros on $(-1, 1)$. The relation~\eqref{Gegenbauer} implies $h$ has $n$ simple zeros on $\mathbb T$, as claimed. 
\end{proof}

The condition $\lambda > -1/2$ in Proposition~\ref{zeros-circle} is sharp. 
Driver and Duren~\cite{DriverDuren-CA} present a detailed picture of the roots of $h$ when $\lambda < -1/2$, in which case they are no longer  contained in $\mathbb T$. 

\begin{lemma}\label{lem:zeros-disk} Fix $n\in \mathbb N$ and a real nonzero number $\nu > -1$. 
All zeros of the hypergeometric polynomial $p(z) = F(-n, \nu+2; -n+1-\nu; z)$  lie in the open unit disk $\mathbb D$. 
\end{lemma}

\begin{proof} We distinguish two cases. If $\nu>-1/2$, Lemma~\ref{zeros-circle}  implies that all zeros of the polynomial $q(z) =  F(-n-2, \nu; -n - 1 - \nu; z)$ are simple and lie on the unit circle $\mathbb T$. The Gauss-Lucas theorem~\cite[\S2.1.6]{Sheil-Small} implies that every derivative of $q$  has all its zeros in $\mathbb D$. The derivative formula for hypergeometric functions~\cite[(2.5.1)]{AndrewsAskeyRoy} yields
\[
q''(z) = \frac{(-n-2)_2 (\nu)_2}{(-n-1-\nu)_2} p(z) 
\]
proving the claim in this case.  

In the case $-1<\nu<0$ we use a different approach, based on the coefficients of $p$:
\[
p(z)= \sum_{k=0}^n 
\frac{(-n)_k (\nu+2)_k }{(-n+1-\nu)_k}  \frac{z^k}{k!}
=: \sum_{k=0}^n a_k z^k.
\] 
When $0\le k\le n-1$ the coefficient  $a_k$ is positive because both $(-n)_k$ and $(-n+1-\nu)_k$ have the same sign as $(-1)^k$. However, $a_n = \frac{(\nu+2)_n }{(\nu)_n} < 0$. By the Chu-Vandermonde identity~\cite[Corollary 2.2.3]{AndrewsAskeyRoy}
\begin{equation}\label{p(1)negative}
\sum_{k=0}^n a_k = 
p(1) = \frac{(-2\nu-n-1)_n}{(-n+1-\nu)_n} = \frac{(2\nu+2)_n}{(\nu)_n} < 0.
\end{equation} 
where the last inequality follows from $2\nu+2>0$ and $-1<\nu < 0$. Therefore, $|a_n|>\sum_{k=0}^{n-1} |a_k|$, which by Rouch\'e's theorem implies that $p$ has all zeros in $\mathbb D$. 
\end{proof}

The main result of this section is a nonlinear relation between several hypergeometric functions which provides a concise formula for the Wronskian determinant of two such functions. 

\begin{lemma}\label{lem:identity} Let $a,b,c\in \mathbb C$ be such that $c = a-b+1$. If $a$ is a negative integer, we assume that $c \notin \{a, a+1, \dots, 0, 1\}$. Otherwise, assume $2-c\notin \mathbb N$. These conditions ensure that the following hypergeometric functions can be defined according to Definition~\ref{def:hyper}: 
\begin{align*}
f(z) & =F(a,b+1;c+1;z) \\
g(z) &=F(a,b-1;c-1;z)\\
h(z) & =F(a,b;c;z)
\end{align*}
The following identity holds for all $z\in \mathbb C$ when $a$ is a negative integer, and for $z\in \mathbb D$ otherwise. 
\begin{equation}\label{wronskian} 
z(fg' - f'g) = c (fg-h^2)   
\end{equation}
\end{lemma}

Our proof of Lemma~\ref{lem:identity} involves  two of the standard linear relations between contiguous hypergeometric functions (\cite[\S2.5]{AndrewsAskeyRoy} or \cite[\S15.5]{DLMF}). The first is stated in~\cite[(15.5.15)]{DLMF} as 
\begin{equation}\label{DLMF.5.15}
(c-a-1)F(a,b;c;z)
+aF(a+1,b;c;z)
-(c-1)F(a,b;c-1;z) = 0
\end{equation}
and the second is~\cite[(15.5.16)]{DLMF}
\begin{equation}\label{DLMF.5.16}
c(1-z) F(a,b;c;z)
-c F(a-1,b;c;z)
+(c-b)z F(a,b;c+1;z)
= 0.
\end{equation}

Two related formulas involve the derivative of $F$: they appear as (2.5.6) and (2.5.7) in~\cite{AndrewsAskeyRoy}. 
\begin{equation} \label{key relation 1}
z \frac{d}{dz} F(a, b; c; z)  = b\big(F(a, b+1; c; z) - F(a, b; c; z))    
\end{equation}
\begin{equation} \label{key relation 2}
z \frac{d}{dz} F(a, b; c; z)  = (c-1)\big( F(a, b; c-1; z) - F(a, b; c; z)\big)  
\end{equation}

\begin{proof}[Proof of Lemma~\ref{lem:identity}] 
Introduce two more hypergeometric functions:
\begin{align*}
\tilde f(z) & =F(a,b+1;c;z);\\
\tilde g(z) & =F(a,b;c-1;z). 
\end{align*}
Interchanging the symmetric parameters $a$ and $b$ in~\eqref{DLMF.5.15}, we find that
\begin{equation} \label{15.5.15}
(c-1-b)h+b\tilde f -(c-1)\tilde g=0.    
\end{equation}
If we use~\eqref{DLMF.5.16} in a similar way, replacing $c$ by $c-1$, the result is   
\begin{equation}  \label{15.5.16.1}
(c-1)(1-z)\tilde g-(c-1)g+(c-1-a)zh=0.    
\end{equation}
Another consequence of~\eqref{DLMF.5.16} is 
\begin{equation}   \label{15.5.16.2}
c(1-z)\tilde f-ch+(c-a)zf=0.
\end{equation}

Recalling that $c-1 = a-b$, multiplying~\eqref{15.5.15} by $(1-z)$ and using ~\eqref{15.5.16.1} and ~\eqref{15.5.16.2} to eliminate $\tilde f$ and $\tilde g$, we obtain
\begin{equation} \label{fgh linear 1}
(c-1)h+(b-a)zh-\frac{c-a}{c}bzf-(c-1)g=0 
\end{equation}
which simplifies to 
\begin{equation}    \label{fgh linear 2}
g=(1-z)h+\frac{b(b-1)}{c(c-1)}zf.    
\end{equation}

Using ~\eqref{fgh linear 2} to eliminate $g$, we find that 
\begin{equation}  \label{clear g}
z(fg' - f'g) - cfg = (cz-c-z)fh-\frac{b(b-1)}{c}zf^2-z(1-z)f'h+z(1-z)fh'. 
\end{equation}
The derivative formulas ~\eqref{key relation 1}--\eqref{key relation 2} allow us to eliminate derivatives from~\eqref{clear g} by using the identities $zh'=b(\tilde f-h)$ and 
$zf'=c(\tilde f-f)$. This leads to
\begin{equation}  \label{fhf'h'}
z(fg' - f'g) - cfg = (bz-b-z)fh-\frac{b(b-1)}{c}z f^2+(1-z)\tilde f(bf-ch).
\end{equation}
Finally, we use~\eqref{15.5.16.2} to eliminate $\tilde f$ from~\eqref{fhf'h'}:
\begin{equation}\label{elim-F}
\begin{split}
  (1-z)\tilde f(bf-ch)
 &= \left( h - \frac{c-a}{c}zf\right)(bf-ch) 
\\ &= bfh -ch^2 - \frac{c-a}{c}bzf^2 + (c-a)zfh . 
\end{split} \end{equation}
Plugging~\eqref{elim-F} into~\eqref{fhf'h'} and using the relation $c-a = 1-b$, we arrive at   
\[z(fg' - f'g) - cfg = -ch^2\]
which is the desired identity~\eqref{wronskian}.
\end{proof}

\section{Hypergeometric Blaschke products}\label{sec:hyper-Blaschke}

\begin{definition}\label{def:hyper-Blaschke} A Blaschke product is called a \emph{hypergeometric Blaschke product} if it has the form $B(z) = p(z)/q(z)$ where $p(z) = F(-n,b;c;z)$ is a hypergeometric polynomial with all zeros contained in $\mathbb D$, and $q(z) = z^n \overline{p(1/\bar z)}$ is its conjugate-reciprocal polynomial. We use the notation $B(z) = B(-n, b; c;z)$ for such products.
\end{definition}

Our main focus will be on the two-parameter family of hypergeometric Blaschke products provided by Lemma~\ref{lem:zeros-disk}.

\begin{lemma}\label{lem:hyper-Blaschke} 
A hypergeometric Blaschke product can be written as 
\begin{equation}\label{hyper-Blaschke-recip}
B(-n,b;c;z) = 
(-1)^n \frac{(c)_n}{(b)_n}
\frac{F(-n,b;c;z)}{F(-n, 1-c-n; 1-b-n; z)}    
\end{equation}
provided that the parameters $b, c$ are real. 
\end{lemma}

\begin{proof} Since the coefficients of the numerator $p$ of $B$ are real, its denominator is the reciprocal polynomial $q(z) = z^n p(1/z)$. Applying a fractional linear transformation of the variable in $F$, see ~\cite[(15.8.6)]{DLMF}, we obtain 
\begin{equation}\label{DLMF15.8.6}
z^n F(-n, b; c; 1/z) =
(-1)^n \frac{(b)_n}{(c)_n}  
F(-n, 1-c-n; 1-b-n; z)
\end{equation}
which proves~\eqref{hyper-Blaschke-recip}.
\end{proof}

For a general finite Blaschke product $B$, the values of $|B'|$ on $\mathbb T$ can be expressed as a rational function and also as a sum of several instances of the \emph{Poisson kernel}
\begin{equation}\label{def-Poisson-kernel}
P(a, z) = \frac{1-|a|^2}{|z-a|^2},\quad a\in \mathbb D, \ z\in \mathbb T.
\end{equation}

\begin{lemma} (\cite[Lemma 3.4]{Ellipses}, \cite[Corollary 3.4.9]{GarciaMashreghiRoss})  \label{lem:poisson}
For any finite Blaschke product $B$ and any $z\in \mathbb T$ we have 
\begin{equation}\label{poisson-identity}
|B'(z)| = \frac{zB'(z)}{B(z)} = \sum_{k=1}^n P(a_k, z)    
\end{equation}
where $a_1, \dots, a_n$ are the zeros of $B$. 
\end{lemma}

\begin{theorem}\label{thm:extreme-values} Suppose $n\in \mathbb N$ and $\nu > -1$, $\nu \ne 0$. The hypergeometric Blaschke product $B(z): = B(-n, \nu+2; -n-\nu+1; z)$ has the following properties: 
\begin{equation}\label{extreme-1}
    M(B) = n+\nu, \quad m(B) = n/(\nu+1) \quad \text{if } \nu>0 ,
\end{equation}
\begin{equation}\label{extreme-2}
    M(B) = n/(\nu+1), \quad m(B) = n+\nu \quad \text{if } -1< \nu < 0 .
\end{equation}
Both extremes are attained within the set $\mathcal E := \{z\in \mathbb T\colon zB(z) = 1\}$. Specifically, $|B'(z)| = n+\nu$ holds on $\mathcal E\setminus \{1\}$ and $|B'(z)| = n/(\nu+1)$ holds when $z=1$. 
\end{theorem}

\begin{proof} Let  $f(z) = F(-n, \nu+2; -n-\nu+1; z)$ and $g(z) = F(-n, \nu; -n-\nu-1; z)$. Lemma~\ref{lem:hyper-Blaschke} shows that
\begin{equation}\label{hB1}
B(z) = \kappa \frac{f(z)}{g(z)}
,\quad 
\text{where }\kappa = (-1)^n \frac{(-n-\nu+1)_n}{(\nu+2)_n} 
= \frac{\nu(\nu+1)}{(n+\nu)(n+\nu+1)} .
\end{equation} 
For $z\in \mathbb T$,  Lemma~\ref{lem:poisson} yields 
\begin{equation}\label{logder B}
|B'(z)| = \frac{zB'(z)}{B(z)} = \frac{z(f'g-fg')}{fg}
\end{equation}
where the denominator can be rewritten as
\begin{equation}\label{rewrite denom}
f(z)g(z) = \kappa 
z^n f(z) \overline{f(1/\bar z)} = \kappa z^n |f(z)|^2.
\end{equation}
From~\eqref{logder B} it follows that
\begin{equation}\label{subtract M}
n+\nu - \frac{zB'(z)}{B(z)} 
 = \frac{(n+\nu)fg - z(f'g-fg')}{fg} .
\end{equation}
Using Lemma~\ref{lem:identity} with $a=-n$,  $b =\nu+1$, and $c = a-b+1 = -n-\nu$, we obtain
\begin{equation}\label{key_step}
(n+\nu)fg - z(f'g-fg') = (n+\nu) h^2
\end{equation}
where $h(z) = F(-n, b, c; z)$. Another useful identity is~\eqref{fgh linear 2} from the proof of Lemma~\ref{lem:identity}, which simplifies to 
\begin{equation}\label{fgh-simplified}
g = (1-z) h + \kappa z f.    
\end{equation}
Because of~\eqref{fgh-simplified} we have 
\begin{equation}\label{extremal-set}
\mathcal E = \{z\in \mathbb T\colon 
\kappa z f(z) = g(z)\} = \{z\in \mathbb T\colon 
(1-z) h(z) = 0\} .
\end{equation}
Thus, the set where $zB(z)=1$ consists of the zeros of polynomial $h$ with an extra point $1$.

By Proposition~\ref{zeros-circle}, all zeros of $h$ are on the unit circle $\mathbb T$. In other words, $h$ has the same zeros as its conjugate-reciprocal polynomial, which implies  $h(z) = \alpha z^n \overline{h(1/\bar z)}$ for some unimodular constant $\alpha$. The hypergeometric form of $h$ shows that its coefficients are positive, thus  $h(z) = z^n \overline{h(1/\bar z)}$. This leads to the identity 
\begin{equation}\label{h ge 0}
\frac{h(z)^2}{z^n} 
= \frac{z^n \overline{h(1/\bar z)} h(z) }{z^n} = |h(z)|^2,\quad z\in \mathbb T . 
\end{equation}
From~\eqref{rewrite denom}, ~\eqref{subtract M}, \eqref{key_step}, and \eqref{h ge 0} we conclude that
\begin{equation}\label{finally}
\begin{split}
n+\nu - \frac{zB'(z)}{B(z)}
& =\frac{(n+\nu)h(z)^2}{f(z)g(z)} \\
& = 
\frac{(n+\nu)h(z)^2}{\kappa  z^n |f(z)|^2}\\ & = 
\frac{(n+\nu)^2(n+\nu+1)}{\nu(\nu+1)}
\frac{|h(z)|^2}{|f(z)|^2}.
\end{split}
\end{equation} 
The identity~\eqref{finally} shows that for $z\in \mathbb T$,
\begin{equation}\label{extreme-B}
\nu( n + \nu - |B'(z)|) \ge 0
\end{equation}
with equality attained precisely at the zeros of $h$. The relation~\eqref{ellipses2} with $\lambda = 1$ implies that 
\[
1 = \sum_{z\in \mathcal E} \frac{1}{|B'(z)|+1}
= \frac{n}{n+\nu+1} + \frac{1}{|B'(1)| + 1}
\]
from where it follows that $|B'(1)| = n/(\nu+1)$.

To complete the proof of~\eqref{extreme-1}--\eqref{extreme-2}, we consider two cases. If $\nu>0$, then $M(B) = n+\nu $ by~\eqref{extreme-B}. The left-hand side of~\eqref{thm:main-ineq} implies that $m(B)\ge n/(\nu+1)$. Since $|B'(1)| = n/(\nu+1)$, it follows that $m(B) = n/(\nu+1)$.

If $-1<\nu<0$, then $m(B) = n+\nu$ by~\eqref{extreme-B}. The right-hand side of~\eqref{thm:main-ineq} implies that $M(B)\le n/(\nu+1)$. Since $|B'(1)| = n/(\nu+1)$, it follows that $M(B) = n/(\nu+1)$.
\end{proof}

\section{Blaschke products with prescribed extrema of the derivative} \label{sec:sharpness} 

In this section we prove Theorem~\ref{thm:main}~(ii), that is, the existence of Blaschke products $B$ with prescribed extrema of $|B'|$ on the unit circle. The proof requires a lemma.

\begin{lemma}\label{lem:monotonicity}
Suppose $a_1, \dots, a_n\in \mathbb D$ and that the number $A:=\max\{ |a_k| \colon k=1,\dots, n\}$ is positive. For $0<\lambda < 1/A $ consider the Blaschke product 
\[
B_{\lambda}(z) = \prod_{k=1}^n \frac{z - \lambda a_k}{1-\lambda \overline{a_k} z}
\] 
Then $M(B_{\lambda})$ is strictly increasing with $\lambda$, and $m(B_{\lambda})$ is strictly decreasing with $\lambda$. 
\end{lemma}

\begin{proof} Recall from~\eqref{poisson-identity} that for $z\in \mathbb T$,
\begin{equation}\label{poisson-relation}
|B_\lambda'(z)| = 
 \sum_{k=1}^n P(\lambda a_k, z).
\end{equation}
For any $\delta\in (0, 1)$, the semigroup property of the Poisson kernel implies
\[
P(\delta \lambda a_k, z) = \int_{\mathbb T} P(\lambda a_k, \zeta) P(\delta, \bar\zeta z)\,\frac{|d\zeta|}{2\pi}. 
\]
Summation over $k$ yields
\begin{equation}\label{averaging}
|B_{\delta\lambda}'(z)| =  
\int_{\mathbb T} |B_{\lambda}'(\zeta)|  P(\delta, \bar\zeta z)\,\frac{|d\zeta|}{2\pi}.
\end{equation}
Since $|B_{\lambda}'(\zeta)|$ is a nonconstant continuous function on $\mathbb T$, the weighted average on the right hand side of~\eqref{averaging} is strictly between $m(B_\lambda)$ and $M(B_\lambda)$. It follows that both $m(B_{\delta\lambda})$ and $M(B_{\delta\lambda})$ are strictly between $m(B_\lambda)$ and $M(B_\lambda)$, which proves both claims of monotonicity. 
\end{proof}

\begin{proof}[Proof of Theorem~\ref{thm:main}~(ii)] 
Suppose $n, m, M$ satisfy~\eqref{thm:main-ineq} and  ~\eqref{thm:main-ineq-trivial}. There are four cases to consider.

Case 1: $M=n$. Then~\eqref{thm:main-ineq} implies $m=n$. The Blaschke product $B_1(z)=z^n$ satisfies $M(B_1)=m(B_1) = n$. 

Case 2: $M>n$ and $m=n/(M-n+1)$, so that equality holds on the left side of~\eqref{thm:main-ineq}. Applying Theorem~\ref{thm:extreme-values} with $\nu=M-n$, we find a hypergeometric Blaschke product $B_2$ for which $M(B_2)=M$ and $m(B_2) =n/(M-n+1)$.

Case 3: $M>n$ and $m=n - 1 + n/M$, so that equality holds on the right side of~\eqref{thm:main-ineq}. Applying Theorem~\ref{thm:extreme-values} with $\nu =n/M - 1 = m-n$, we find a hypergeometric Blaschke product $B_3$ for which $M(B_3)=M$ and $m(B_3) = n - 1 + n/M$.  

Case 4: $M>n$ and strict inequalities hold on both sides of~\eqref{thm:main-ineq}. We handle this case by interpolating between the products $B_2$ and $B_3$ from the preceding cases. Recall that $M(B_2)=M(B_3)=M$ and $m(B_2) < m < m(B_3)$. Let $a_1, \dots, a_n$ be the zeros of $B_2$, and let $b_1, \dots, b_n$ be the zeros of $B_3$. Note that $a_k, b_k\ne 0$ since a hypergeometric function does not vanish at $z=0$. 

For $k=1, \dots, n$ choose some continuous curve $\gamma_k\colon [0, 1]\to \mathbb D\setminus \{0\}$ such that $\gamma_k(0) = a_k$ and $\gamma_k(1)=b_k$.  Define $\Lambda(t) = 1/\max_k |\gamma_k(t)|$. For $0<\lambda < \Lambda(t)$ let $B_{t,\lambda}$ be a Blaschke product  with zeros $\lambda \gamma_k(t)$, $k=1,\dots,n$. 

For each fixed $t$, the function $\lambda\mapsto M(B_{t,\lambda})$ is continuous and strictly increasing by virtue of Lemma~\ref{lem:monotonicity}. The relation~\eqref{poisson-relation} shows that  $M(B_{t,\lambda})\to n$ as $\lambda\to 0+$, and  $M(B_{t,\lambda})\to \infty$ as $\lambda\to \Lambda(t)-$, when at least one of the zeros approaches the boundary of $\mathbb D$. Therefore, there exists a unique value $\lambda(t)$ such that $M(B_{t,\lambda(t)}) = M$.  We claim that the function $t\mapsto \lambda(t)$ is continuous on $[0, 1]$. 

Fix $t\in [0, 1]$ and $\epsilon>0$. Choose some positive numbers $\alpha, \beta$ so that
\[
\max(\lambda(t) - \epsilon, 0)< \alpha < \lambda(t) < \beta < \min(\lambda(t) + \epsilon, \Lambda(t)) 
\]
By the strict monotonicity with respect to $\lambda$, 
\begin{equation}\label{up-low1}
M(B_{t, \alpha}) < M < M(B_{t,\beta}) 
\end{equation}
The continuity of $t\mapsto M(B_{t, \alpha})$ with respect to $t$ implies the existence of $\delta>0$ such that if $s\in [0, 1]$ and $|s-t|<\delta$, both inequalities~\eqref{up-low1} still hold with $t$ replaced by $s$. Therefore, by the monotonicity of $\lambda \mapsto M(B_{t, \lambda})$ for any such $s$ we have $\alpha<\lambda(s)<\beta$, which implies $|\lambda(s) - \lambda(t)| < 2\epsilon$. This proves the continuity of $\lambda$ with respect to $t$. 

In view of the continuity of $t\mapsto \lambda(t)$, the function $t\mapsto m(B_{t,\lambda(t)})$ is continuous and the intermediate value theorem provides $t\in (0, 1)$ such that $m(B_{t,\lambda(t)}) = m$, completing the proof of Theorem~\ref{thm:main}. \end{proof}

\section{Uniqueness of extremal Blaschke products}\label{sec:uniqueness} 

In this section we prove  Theorem~\ref{thm:uniqueness} concerning the uniqueness of Blaschke products $B$ with extreme values of $M(B)=\sup_{\mathbb T}|B'|$   and $m(B)=\inf_{\mathbb T}|B'|$, as introduced in Definition~\ref{def:extremal-Bp}.

\begin{lemma}\label{lem:preparation}
Suppose $B$ is an extremal Blaschke product of degree $n$. Let $\nu$ be as in~\eqref{nu-def}. Then there exists $\lambda \in \mathbb T$ and distinct points $z_0, z_1, \dots, z_{n}\in \mathbb T$ with the following properties:   
\begin{enumerate}[(a)]
    \item $z_k B(z_k) = \lambda$ for $k=0, \dots, n$;
    \item $|B'(z_0)| = n/(\nu+1)$;
    \item $|B'(z_k)| = n + \nu$ for $k=1, \dots, n$.
\end{enumerate}

\end{lemma}

\begin{proof} Case 1: $\nu\ge 0 $. Then $M(B) = n+\nu$ and $m(B) = n/(\nu+1)$, according to~\eqref{nu-prop}. Therefore, there exists $z_0\in \mathbb T$ such that $|B'(z_0)| = n/(\nu+1)$. Let $\lambda = z_0B(z_0)$ and let $z_1, \dots, z_n$ be other solutions of the equation $zB(z) = \lambda$. The identity~\eqref{ellipses2} implies
\begin{equation}\label{ellipses2a}
\sum_{k=1}^{n} \frac{1}{|B'(z_k)| + 1} = 1 - \frac{1}{n/(\nu+1) + 1}  = 
 \frac{n}{n+\nu+1}
\end{equation} 
Since $|B'| \le n+\nu $ on $\mathbb T$, in order for~\eqref{ellipses2a} to hold, it is necessary to have $|B'(z_k)| = n+\nu $ for $k=1, \dots, n$. This concludes the proof of Case~1. 

Case 2: $-1<\nu < 0$, which implies $M(B) =n/(\nu+1)$ and $m(B) = n+\nu$ according to ~\eqref{nu-prop}. Pick $z_0\in \mathbb T$ so that $|B'(z_0)| = M(B) = n/(\nu+1)$. Then define $\lambda, z_1, \dots, z_n$ as in Case~1. The relation~\eqref{ellipses2a} still holds. Since $|B'|\ge n+\nu$ on $\mathbb T$, it follows that $|B'(z_k)| = n+\nu $ for $k=1, \dots, n$. 
\end{proof}

\begin{proof}[Proof of Theorem~\ref{thm:uniqueness}] Let $B$ be an extremal Blaschke product of degree $n$, and let $\nu $ be as in~\eqref{nu-def}. If $\nu = 0$, then $M(B) = m(B) = n$, which means $|B'|\equiv n$ on $\mathbb T$. Hence 
$B(z) = \alpha z^n$ for some $\alpha \in \mathbb T$. From now on, assume $\nu\ne 0$.

By replacing $B$ with $\alpha B(\beta z)$ for appropriate $\alpha, \beta \in \mathbb T$, we can make sure that the conclusion of Lemma~\ref{lem:preparation} holds with $z_0=1$ and $\lambda = 1$. In particular, $B(1) = 1$.  

Let $\widetilde{B}(z) = zB(z)$ and $\mathcal E = \{z_0, \dots, z_{n}\} = \widetilde B^{-1}(1)$. We can write $B$ as $B(z) = p(z)/q(z)$ where $p$ is a polynomial of degree $n$ (not necessarily monic) with all zeros in $\mathbb D$, and $q$ is its conjugate-reciprocal polynomial. We will usually omit the argument $z$ of $p$ and $q$ and other polynomials that appear below.  

For $z\in \mathbb T$ we have 
\begin{equation}\label{psipq}
n+\nu-|B'| = n+\nu-\frac{zB'}{B} = n+\nu-z\frac{p'q-q'p}{pq} = 
\frac{\psi}{pq}
\end{equation}
where 
\begin{equation}\label{psi-def}
 \psi = (n+\nu)pq-z(p'q-q'p)   
\end{equation}
is a polynomial of degree $2n$. Because of~\eqref{psipq}, the rational function $\psi/(pq)$ maps $\mathbb T$ onto a real interval with $0$ as an endpoint. By Lemma~\ref{lem:preparation} this function attains the extreme value $0$ at the points $z_1, \dots, z_n$, which therefore must be zeros of even order for $\psi$. Considering that $\deg \psi = 2n$, we conclude that $\psi =C_1 r^2$ where  $C_1$ is a constant and $r(z)=\prod_{k=1}^n (z-z_k)$.

Since $\widetilde{B}$ maps $\mathcal E$ to $1$, we have $zp= q$ on the set $\mathcal E$. Plugging this relation into~\eqref{psi-def}, we find 
\[\psi(z) = (n+\nu)zp^2-z^2pp'+zpq' = zp[(n+\nu)p-zp'+q'], \quad z\in \mathcal E.\]
Since $zp$ does not vanish on $\mathbb T$, the factor $(n+\nu)p-zp'+q'$ vanishes at $z_1, \dots, z_n$. But it has degree $n$, and therefore   
\begin{equation}   \label{first form}
(n+\nu)p-zp'+q'=C_2 r(z)    
\end{equation}
where $C_2$ is a constant.

Recalling that $\mathcal E = \{1, z_1, \dots, z_n\}$ is the zero set of $zp-q$ and using~\eqref{first form}, we arrive at     
\begin{equation}\label{zp-q}
zp- q=C_3(z-1)[(n+\nu)p-zp'+q']
\end{equation}
where $C_3$ is a constant. Comparison of the leading coefficients in~\eqref{zp-q} leads to $C_3=1/\nu$. Hence
\begin{equation}   \label{key identity}
\nu(zp- q) = (z-1)[(n+\nu)p-zp'+q' ].   
\end{equation}

The identity~\eqref{key identity} provides a linear differential equation relating the polynomials $p$ and $q$. Another such relation can be derived as follows. 

Since the points $z_1, \dots, z_n$ are  double roots of $\psi$, they are also roots of its derivative 
\begin{equation} \label{derivative}
\psi' = (n+\nu-1)p'q+(n+\nu+1)pq'-zp''q+zpq'' .
\end{equation}
Since $q=zp$ at $z_1, \dots, z_n$, plugging this relation in \eqref{derivative} shows that 
\begin{equation}\label{factor-p}
\psi'(z) = p[(n+\nu-1)zp'+(n+\nu+1)q'-z^2p''+zq''], \quad z\in \{z_1, \dots, z_n\}.
\end{equation}
Since $p$ does not vanish on $\mathbb T$, the second factor on the right hand side of~\eqref{factor-p}, which is
\begin{equation}\label{second form}
    (n+\nu-1)zp'+(n+\nu+1)q'-z^2p''+zq''
\end{equation}
must vanish at $z_1, \dots, z_n$.  But~\eqref{second form} is a polynomial of degree $n$ and therefore must be a constant multiple of $r$. The leading coefficient of~\eqref{second form} is $n\nu c_n $ 
where $c_n$ is the leading coefficient of $p$. The leading coefficient of the left hand side of~\eqref{first form} is $\nu c_n$. Thus, 
\begin{equation}  \label{two forms of product}
(n+\nu-1)zp'+(n+\nu+1)q'-z^2p''+zq'' = n[(n+\nu)p-zp'+q'] .
\end{equation}

Our next step is to combine~\eqref{key identity} and~\eqref{two forms of product}, eliminating $q$ and thus obtaining a differential equation for $p$. Specifically, we will show that $p$ satisfies the hypergeometric differential equation 
\begin{equation}  \label{differential equation}
z(1-z)p''-(n+\nu-1)p'-(\nu-n+3)zp'+n(\nu+2)p=0 .
\end{equation}

Expanding~\eqref{key identity} as 
\begin{equation}   \label{key identity reformulation}
z(1-z)p'+ nzp - (n+\nu)p + \nu q+(z-1)q'=0    
\end{equation}
and taking derivative on both sides, we get
\begin{equation} \label{key identity derivative}
z(1-z)p''-(n+\nu-1)p'+(n-2)zp'+np+(\nu+1)q'+(z-1)q''=0    .
\end{equation}
The system of two second-order differential equations~\eqref{two forms of product} and~\eqref{key identity derivative} can be written as  
\begin{align}  \label{system1}
(\nu+1)q'+zq'' & =n(n+\nu)p-(2n+\nu-1)zp'+z^2p'' =:\Phi \\      
(\nu+1)q'+(z-1)q'' &
\label{system2}
=-np+(n+\nu-1)p'-(n-2)zp'-z(1-z)p'' =: \Psi    
\end{align}
Subtracting~\eqref{system2} from~\eqref{system1} shows that $q'' = \Phi-\Psi$, and then~\eqref{system1} yields  $(\nu+1)q' = (1-z)\Phi  +z\Psi$. Differentiating the latter, we obtain
\[
(1-z)\Phi'  +z\Psi'  - \Phi + \Psi 
= (\nu+1)q'' = (\nu+1)(\Phi-\Psi) .
\]
Hence
\begin{equation}\label{diff-eq}
(1-z)\Phi'  +z\Psi'  + (\nu + 2)(\Psi-\Phi) = 0.
\end{equation} 
Recalling the definitions of $\Phi$ and $\Psi$~\eqref{system1}--\eqref{system2}, we see that~\eqref{diff-eq} is a linear differential equation for $p$.  When simplified, it becomes the hypergeometric equation ~\eqref{differential equation}.

At this point one could appeal to the classification of solutions of the hypergeometric equation~\cite[\S2.3]{AndrewsAskeyRoy} to conclude that $p$ must be a constant multiple of $F(-n,\nu+2;-n-\nu+1;z)$. But with this approach, the case when $-n-\nu+1$ is a negative integer has to be treated separately. To avoid this complication, we work directly with the coefficients of polynomial  $p(z)=\sum_{k=0}^n c_k z^k$. Plugging this sum into ~\eqref{differential equation} and equating the coefficient of $z^k$,  $0\le k\le n-1$, to zero, we find that
\[k(k+1)c_{k+1}-k(k-1)c_k -(n+\nu-1) (k+1)c_{k+1}-(\nu-n+3)kc_k+n(\nu+2)c_k=0\]
hence 
\begin{equation}\label{coeff-recursion}
c_{k+1} = \frac{(k-n)(k+\nu+2)}{(k+1)(k-n-\nu+1)}
c_{k}.
\end{equation} 
The denominator in~\eqref{coeff-recursion} is never zero because $\nu\in (-1, 0)\cup (0, \infty)$ and $k\le n-1$.  Comparing~\eqref{coeff-recursion} to the definition of hypergeometric function~\eqref{hyper-series}, we see that 
\[
p(z) = c_0 F(-n,\nu+2;-n-\nu+1;z).
\]
The normalization $B(1)=1$ implies that $c_0/\overline{c_0} = 1$, hence $c_0$ is real. Dividing $p$ by the real constant $c_0$ does not change the Blaschke product $B=p/q$ because the conjugate-reciprocal polynomial $q$ gets divided by the same constant. Thus, $B$ has the form claimed in Theorem~\ref{thm:uniqueness}. 
\end{proof} 

\section{Zeros of extremal Blaschke products}\label{sec:zeros}

Although Theorem~\ref{thm:uniqueness} provides an explicit formula for extremal Blaschke products, more insight can be gained from examples which unpack the definition of the hypergeometric function $F$ and display the structure of zeros and the behavior of $|B'|$ on the unit circle. 

\begin{example}\label{example:deg2} Recall from Corollary~\ref{cor:deg2} that for Blaschke products $B$ with $\deg B=2$, the inequality $M(B)\le 3$ implies $m(B)\ge 1$. By Theorem~\ref{thm:uniqueness}, there is a unique (up to rotation) Blaschke product with $\deg B = 2$,  $M(B)=3$, and $m(B)=1$, namely
\[
B(z) = \frac{6z^2 + 3z + 1}{z^2 + 3z + 6} .
\]
The zeros of $B$ are $(-3 \pm i\sqrt{15})/12$, which illustrates the fact that while extremal Blaschke products have relatively simple coefficients (positive rational numbers) their zeros do not have a simple form. 
\end{example}

To understand the role of the location of zeros, recall that the boundary values of $|B'|$ are given by the Poisson kernel sum~\eqref{poisson-identity} involving the zeros of $B$. In order to minimize $m(B)$ for a fixed value of $M(B)$, it is logical to distribute the zeros so they have a large gap, where small values of $|B'|$ are found, but are uniformly spaced otherwise, avoiding very large values of $|B'|$.   

Extremal products of the second kind, which maximize $m(B)$ for a given $M(B)$, have a less intuitive placement of zeros. One might think that the best way to keep the sum~\eqref{poisson-identity} nearly constant around the unit circle is to distribute the zeros $a_1, \dots, a_n$ uniformly on a circle concentric to $\mathbb T$. But this is not so. 

\begin{example}\label{example:deg15} Let $B_1$ be an extremal Blaschke product of the first kind with $n=15$ and $\nu=5$, which means $M(B_1) = 20$ and $m(B_1) = 2.5$. Figure~\ref{fig:deg15} (left) shows the zeros of $B_1$, with the unit circle for reference.

Let $B_2$ be an extremal Blaschke product of the second kind with $n=15$ and $\nu=-1/4$, which means $M(B_2) = 20$ and $m(B_2) = 14.75$. The zeros of $B_2$ are shown on the right in Figure~\ref{fig:deg15}. The principal difference is that $B_2$ has a positive zero. 
\end{example}

\begin{figure}[h]
    \centering
    \includegraphics*[trim={120 100 120 100}, clip, width=0.8\textwidth]{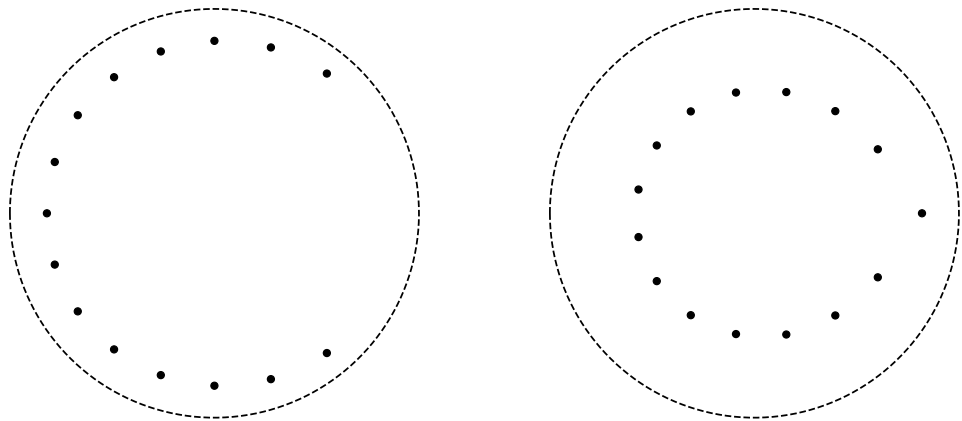}
    \caption{Zeros of the extremal Blaschke products of   Example~\ref{example:deg15}}
    \label{fig:deg15}
\end{figure}

\begin{remark} For $-1<\nu<0$ the polynomial $p(z) = F(-n, \nu+2; -n-\nu+1; z)$, which is the numerator of an extremal Blaschke products of the second kind, has a real root in the interval $(0, 1)$. Indeed, $p(0)=1$ and $p(1)<0$ by virtue of~\eqref{p(1)negative}.
\end{remark} 
 
The behavior of $|B'(e^{it})|$, $-\pi\le t\le \pi$, is shown on 
Figure~\ref{fig:deg15der} for both products introduced in Example~\ref{example:deg15}. Both plots share the vertical axis, demonstrating that $M(B_1)=M(B_2) = 20$ while the minimum values are quite different. In both cases the mean value of $|B'|$ on $\mathbb T$ is equal to $15$, the degree of the product. 
 
\begin{figure}[h]
    \centering
    \includegraphics*[trim={50 20 60 20}, clip, width=0.7\textwidth]{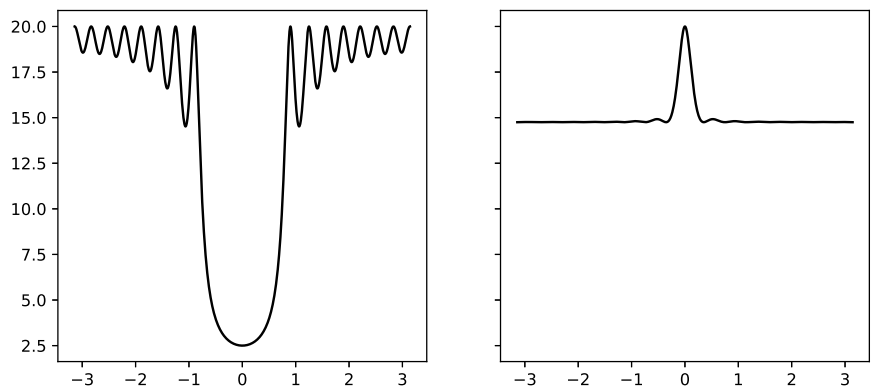}
    \caption{Derivative of the extremal Blaschke products of   Example~\ref{example:deg15}}
    \label{fig:deg15der}
\end{figure}
 
The fact that the mean value of $|B'|$ is constrained by the degree of $B$ explains why a symmetric placement of zeros  fails to maximize $m(B)$ for a given $M(B)$. For a Blaschke product $B$ with symmetric zeros, $|B'|$ attains its maximum $n$ times, and these multiple maxima must be offset by smaller values of $|B'|$ elsewhere. Indeed, an easy computation shows that a product with symmetric zeros,
\[
B(z) = \frac{z^n+a^n}{1 + a^n z^n}, \quad 0<a<1,
\]
 has 
 \[
|B'(z)| = n\frac{1-a^{2n}}{|z^n+a^n|}, 
\quad M(B) = n\frac{1+a^{n}}{1 - a^n}, \quad m(B) = n\frac{1-a^{n}}{1 + a^n}, 
 \]
 and therefore $m(B) = n^2/M(B)$, which is not extremal in the sense of Definition~\ref{def:extremal-Bp}.
 
 \section*{Acknowledgment}
 
 We thank Pamela Gorkin for bringing the papers~\cite{GauWu} and~\cite{GorkinRhoades} to our attention. 
 
\bibliographystyle{plain}
\bibliography{references.bib}

\end{document}